\documentclass[11pt,twoside,reqno]{amsart}
\allowdisplaybreaks
\usepackage{amsmath,amstext,amssymb,epsfig}
\usepackage{graphicx}
\usepackage{mathrsfs}
\calclayout
\makeatletter
\renewcommand{\@seccntformat}[1]{\bf\csname the#1\endcsname.}
\renewcommand{\section}{\@startsection{section}{1}
	\z@{.7\linespacing\@plus\linespacing}{.5\linespacing}
	{\normalfont\upshape\bfseries\centering}}
\renewcommand{\@biblabel}[1]{\@ifnotempty{#1}{#1.}}
\makeatother
\theoremstyle{plain}
\newtheorem{thm}{Theorem}[section]

\newtheorem{cor}[thm]{Corollary}

\theoremstyle{definition}

\newtheorem{defn}[thm]{Definition}

\usepackage{cancel}
\usepackage[parfill]{parskip}
\usepackage[german]{varioref}
\usepackage[all]{xy}
\usepackage{color}

\def\A{{\mathcal A}}

\def \>{\succ}
\def \<{\prec}

\begin{document}	
	\title[Ahmed Zahari Abdou\textsuperscript{1}, Bouzid Mosbahi\textsuperscript{2}]{Computational Approaches to Derivations and Automorphism Groups of Associative Algebras}
	\author{Ahmed Zahari Abdou\textsuperscript{1}, Bouzid Mosbahi\textsuperscript{2}}

		\address{\textsuperscript{2}
IRIMAS-Department of Mathematics, Faculty of Sciences, University of Haute Alsace, Mulhouse, France}
		\address{\textsuperscript{1}Department of Mathematics, Faculty of Sciences, University of Sfax, Sfax, Tunisia}
\email{\textsuperscript{1}abdou-damdji.ahmed-zahari@uha.fr}
\email{\textsuperscript{2}mosbahi.bouzid.etud@fss.usf.tn}
	
	\keywords{Associative algebras, Structure constants, automorphism, derivation}
	\subjclass[2010]{16 D70}
	
	\date{\today}
    
	\begin{abstract}
This paper focuses on the derivations and automorphism groups of certain finite-dimensional associative algebras over the field of complex numbers. Using classification results for algebras of dimensions two, three, and four, along with computational tools like Mathematica and Maple, we offer detailed descriptions of the derivations and automorphism groups for these algebras. Our analysis of these groups helps to uncover important structural features and symmetries in low-dimensional associative algebras.
     \end{abstract}

\maketitle \section{ Introduction}\label{introduction}
Associative algebras are key structures in algebra, with deep connections to various other algebraic systems, such as Lie and Jordan algebras. These algebras are crucial for understanding fundamental mathematical concepts and have widespread applications in areas like physics, geometry, and computer science. Among the most studied associative algebras are those that are finite-dimensional. The classification of these algebras, especially those of low dimensions, has been an area of active research for many years.

The classification of finite-dimensional associative algebras has a long history. The first significant contributions came from Peirce \cite{1,2}, who laid the groundwork for the study of algebras. Later, in 1916, Hazlett \cite{3} provided a classification of nilpotent algebras of dimension up to four. In 1979, Mazzola \cite{4} advanced the field by studying five-dimensional algebras, using both algebraic and geometric methods. Despite these notable advancements, the complete classification of finite-dimensional associative algebras remains unsolved.

The study of associative algebras dates back even further to the work of famous mathematicians such as Hamilton, who discovered quaternions, and Cayley, who developed matrix theory. Over the years, other mathematicians, including Peirce, Clifford, Weierstrass, Dedekind, and Frobenius, made significant contributions to the theory of associative algebras. At the turn of the 20th century, mathematicians like Molin and Cartan made substantial progress in the study of semisimple algebras, particularly over real and complex fields.

This paper focuses on an important aspect of associative algebras: their derivations and automorphism groups. These concepts are central to understanding the symmetries and internal structure of algebras. The derivations of an algebra provide insight into how elements of the algebra can be differentiated in a way that preserves the algebra's structure. The automorphism group of an algebra describes the symmetries of the algebra, revealing how the algebra can be mapped onto itself while preserving its structure.

Using the classification results for associative algebras of dimensions two, three, and four \cite{10}, we provide explicit descriptions of the derivations and automorphism groups for these low-dimensional algebras. Computational tools such as Mathematica and Maple play a crucial role in performing the necessary calculations and visualizing these groups, making it possible to uncover important properties and symmetries in these algebras.

The paper is structured as follows. Section 2 introduces the key definitions and concepts needed to understand the topic. In Sections 3 and 4, we explore the derivations and automorphism groups of low-dimensional associative algebras, respectively. Throughout the paper, we focus on finite-dimensional associative algebras over the field of complex numbers, $\mathbb{C}$.

\section{ Prelimieries}

\begin{defn}\label{a}
An \textit{associative algebra} $A$ is a vector space over $\mathbb{C}$ equipped with bilinear map 	$\cdot: A \times A :\longrightarrow A$ satisfying
the associative law:
\begin{align*}
(x\cdot y)\cdot z&=x\cdot(y\cdot z), \quad \forall x, y, z \in A.
\end{align*}
\end{defn}

\begin{defn}\label{d1}
A linear transformation $D$ of an associative algebra $A$ is called a derivation if for any $x, y \in A$
\begin{align*}
D(x\cdot y) &= D(x)\cdot y + x\cdot D(y).
\end{align*}
The set of all derivations of an associative algebra $A$ is denoted by $Der(A)$.
\end{defn}
\begin{defn}
Let $H$ be a non empty subset of $A$. The set $Z_A(H) = \{x \in A : x.H = H.x = 0\}$.
is said to be centralizer of H in $A$.
It must however be noted that $Z_A(A) = Z(A)$, the center of $A$.
\end{defn}
\begin{defn}
Let $\phi \in End(A)$. If $\phi(A) \subseteq  Z(A)$ and $\phi(A^2) =0$, then $\phi$ is called a central derivation.
The set of all central derivations is denoted by $C(A)$.
\end{defn}

\begin{defn}
Let $(A_1, \cdot)$ and $(A_2, \ast)$ be two associative algebras over $\mathbb{C}$. A homomorphism between $A_1$ and $A_2$ is a $\mathbb{C}$-linear mapping $f : A_1 \longrightarrow A_2$ such that
\begin{align*}
f(x\cdot y) &= f(x) \ast f(y) \; \forall \; x, y \in A_1.
\end{align*}
The set of all homomorphism from $A_1$ to $A_2$ is denoted by
$Hom_{\mathbb{C}}(A_1,A_2)$. If $A_1 = A_2 = A$, then it is an associative algebra with respect to composition operation and denoted by $Hom_{\mathbb{C}}(A)$.
The linear mapping associated with $Hom_{\mathbb{C}}(A)$ is called an endomorphism. A bijective homomorphism is called isomorphism and the corresponding algebras are said to be isomorphic.
\end{defn}

\begin{defn}
	\textbf{Automorphism} of an algebra $(A,.)$ over a field $\mathbb{C}$ is a $\mathbb{C}$-linear transformation $f :A\longrightarrow A$
	preserving the operation of $A$, i.e
	
	\begin{itemize}
		\item $f(\alpha x+\beta y)=\alpha f(x)+\beta f(y)$ for any $\alpha, \beta\in \mathbb{F}$ and $x, y\in \mathcal{\A}$
		\item $f(x.y)=f(x).f(y)$ for any $x, y\in A$.
	\end{itemize}
	The set all automorphism of an algebra $A$ froms a group with respect to the composition. The automorphismgroup of $A$ is denoted by $Aut(A).$
\end{defn}

By utilizing the classification results \cite{10} and solving the equations, the derivations and automorphism groups can be determined and summarized.

\section{ An Algorithmic Approach to Derivations}

Let $\{e_1, e_2, \ldots, e_n\}$ be the basis of $A$, an $n$-dimensional complex associative algebra. The components of $e_i \cdot e_j$, for $i, j = 1, 2, \ldots, n$, define the structure constants of $A$ on the basis $\{e_1, e_2, \ldots, e_n\}$. Specifically, if  
\[
e_i \cdot e_j = \sum_{k=1}^n \gamma^k_{ij} e_k,
\]
then the set of structure constants of $A$ is denoted by  
\[
\{\gamma^k_{ij} \mid i, j, k \leq n\}.
\]  
These structure constants are defined over the field of complex numbers $\mathbb{C}$.

To determine the derivations of $A$, consider $D = (d_{ij})_{i,j=1,2,\ldots,n}$ as the matrix representation of a derivation $d$ with respect to the basis $\{e_1, e_2, \ldots, e_n\}$. A derivation satisfies the Leibniz rule:  
\[
D(e_i \cdot e_j) = D(e_i) \cdot e_j + e_i \cdot D(e_j).
\]  
Using the structure constants $\{\gamma^k_{ij}\}$, this translates into the following system of equations:  
\[
\sum_{k=1}^n \gamma^k_{ij} d_{tk} = \sum_{k=1}^n \left( d_{ki} \gamma^t_{kj} + d_{kj} \gamma^t_{ik} \right),
\]  
for $1 \leq i, j, t \leq n$.

This system provides a method to compute the derivations of $A$ explicitly.

\begin{thm}\label{thm1}
The \textbf{drivation} of $2$-dimensional associative algebras have the
following form:\\
$$
D(As_2^1) :\left(\begin{array}{cccc}
d_{11}&0\\
d_{21}&2d_{11}\\
\end{array}
\right)
\quad;
\quad
D(As_2^2) :\left(\begin{array}{cccc}
0&0\\
0&d_{22}\\
\end{array}
\right)
\quad;
\quad
D(As_2^3):\left(\begin{array}{cccc}
0&0\\
0&d_{22}\\
\end{array}
\right)
$$

$$
D(As_2^4):\left(\begin{array}{cccc}
0&0\\
d_{21}&d_{22}\\
\end{array}
\right)\quad;
\quad
D(As_2^5):\left(\begin{array}{cccc}
0&0\\
0&0
\end{array}
\right).
$$
\end{thm}
\begin{proof}

Let \( \{e_1, e_2\} \) be a basis of the two-dimensional associative algebra \( As_2^1 \). Using the definition of a derivation, as in Definition \ref{d1}, we have:
\[
D(e_1 \cdot e_1) = D(e_2) = 2d_{11}e_2.
\]

Next, applying the derivation property:
\[
D(e_1) \cdot e_1 + e_1 \cdot D(e_1) = (d_{11}e_1 + d_{21}e_2) \cdot e_1 + e_1 \cdot (d_{11}e_1 + d_{21}e_2) = 2d_{11}e_2.
\]

Thus, it follows that \( D(x \cdot y) = D(x) \cdot y + x \cdot D(y) \) for all \( x, y \in \{e_1, e_2\} \).

Hence, the matrix representation of the derivation \( D \) is given by:
\[
D = \begin{pmatrix}
d_{11} & d_{12} \\
d_{21} & d_{22} \\
\end{pmatrix}
\quad \text{where} \quad \det(D) \neq 0.
\]

The derivation of the basis elements can be written as:
\[
D(e_1) = d_{11}e_1 + d_{21}e_2, \quad D(e_2) = d_{12}e_1 + d_{22}e_2.
\]

Since for two-dimensional algebras, we have \( e_i \cdot e_j \) with \( i, j = 1, 2 \), we must consider four cases. By using the algebra \( As_2^1 \) where \( e_1e_1 = e_2 \), we find that \( d_{12} = 0 \). As a result, we obtain the following matrix representation for the derivation:
\[
D(As_2^1) = \begin{pmatrix}
d_{11} & 0 \\
d_{21} & 2d_{11} \\
\end{pmatrix}.
\]

Using similar methods for all algebras in Theorem \ref{thm1}, we obtain the derivations of all two-dimensional algebras as shown above.
\end{proof}

\begin{thm}\label{thm2}
The \textbf{drivation} of $3$-dimensional associative algebras have the following form:

$
D(As_3^1) :
\begin{array}{ll}
\left(\begin{array}{cccc}
d_{11}&0&0\\
d_{21}&d_{22}&d_{23}\\
0&0&d_{22}-d_{11}
\end{array}
\right)
\end{array}
 ;
D(As_3^2) :
\begin{array}{ll}
\left(\begin{array}{cccc}
d_{11}&0&0\\
d_{21}&d_{22}&d_{23}\\
0&0&d_{22}-d_{11}
\end{array}
\right)
\end{array}
;
D(As_3^3) :
\begin{array}{ll}
\left(\begin{array}{cccc}
d_{11}&0&0\\
d_{21}&2d_{11}&0\\
d_{31}&2d_{21}&3d_{11}
\end{array}
\right)
\end{array}
$

$
D(As_3^4) :
\begin{array}{ll}
\left(\begin{array}{cccc}
d_{11}&0&0\\
d_{21}&d_{11}+d_{21}&d_{23}\\
0&0&0
\end{array}
\right)
\end{array}
;
D(As_3^5) :
\begin{array}{ll}
\left(\begin{array}{cccc}
d_{11}&0&d_{13}\\
0&d_{22}&d_{23}\\
0&0&0
\end{array}
\right)
\end{array}
;
D(As_3^6) :
\begin{array}{ll}
\left(\begin{array}{cccc}
d_{11}&0&0\\
d_{21}&d_{11}+d_{21}&d_{23}\\
0&0&0
\end{array}
\right)
\end{array}
$

$
D(As_3^7) :
\begin{array}{ll}
\left(\begin{array}{cccc}
d_{11}&d_{12}&-d_{12}\\
0&0&0\\
0&0&0
\end{array}
\right)
\end{array}
;
D(As_3^8) :
\begin{array}{ll}
\left(\begin{array}{cccc}
d_{11}&0&0\\
0&d_{22}&d_{23}\\
0&0&0
\end{array}
\right)
\end{array}
;
D(As_3^9) :
\begin{array}{ll}
\left(\begin{array}{cccc}
d_{11}&0&d_{13}\\
0&d_{22}&0\\
0&0&0
\end{array}
\right)
\end{array}
$

$
D(As_3^{10}) :
\begin{array}{ll}
\left(\begin{array}{cccc}
d_{11}&d_{12}&0\\
d_{21}&d_{22}&0\\
0&0&0
\end{array}
\right)
\end{array}
;
D(As_3^{11}) :
\begin{array}{ll}
\left(\begin{array}{cccc}
d_{11}&0&0\\
d_{21}&d_{11}+d_{21}&0\\
0&0&0
\end{array}
\right)
\end{array}
;
D(As_3^{12}) :
\begin{array}{ll}
\left(\begin{array}{cccc}
d_{11}&0&0\\
d_{21}&2d_{11}&0\\
0&0&0
\end{array}
\right)
\end{array}.
$
\end{thm}

\begin{proof}

Let \( \{e_1, e_2, e_3\} \) be a basis of the three-dimensional associative algebra \( As_3^1 \), where the relations are given by \( e_1 e_3 = e_2 \) and \( e_3 e_1 = e_2 \). Using the definition of a derivation, as in Definition 2.2, we have:

\[
D(e_1 \cdot e_3) = D(e_2) =d_{22}e_2, \quad D(e_1).e_3+e_1.D(e_3)=d_{22}e_2.
\]

Next, applying the derivation property:

\[
D(e_3.e_1)=D(e_2)=d_{22}e_2, \quad  D(e_3).e_1+e_3.D(e_1)=d_{22}e_2.
\]

Thus, it follows that \( D(x \cdot y) = D(x) \cdot y + x \cdot D(y) \) for all \( x, y \in \{e_1, e_2, e_3\} \).

Hence, the matrix representation of the derivation \( D \) is given by:

\[
D = \begin{pmatrix}
d_{11} & 0 & 0 \\
d_{21} & d_{22} & d_{23} \\
0 & 0 & d_{22} - d_{11}
\end{pmatrix}
\quad \text{where} \quad \det(D) \neq 0.
\]

The derivation of the basis elements can be written as:
\[
D(e_1) = d_{11}e_1 + d_{21}e_2, \quad D(e_2) = d_{22}e_1 + d_{23}e_2, \quad D(e_3) = (d_{22} - d_{11})e_3.
\]

Since for the three-dimensional algebra, we have \( e_1 e_3 = e_2 \) and \( e_3 e_1 = e_2 \), we must consider all the corresponding products. By using the relations in \( As_3^1 \), we obtain the following matrix representation for the derivation:

\[
D(As_3^1) = \begin{pmatrix}
d_{11} & 0 & 0 \\
d_{21} & d_{22} & d_{23} \\
0 & 0 & d_{22} - d_{11}
\end{pmatrix}.
\]
Using similar methods for all algebras in Theorem 3.1, we obtain the derivations of all three-dimensional algebras as shown above.
\end{proof}

\begin{thm}\label{thm3}
The \textbf{drivation} of $4$-dimensional associative algebras have the following form:\\

$
D(As_4^{1}) :
\left(\begin{array}{cccc}
d_{11}&0&0&0\\
0&d_{22}&0&0\\
d_{31}&d_{32}&d_{11}+d_{22}&0\\
d_{41}&d_{42}&0&d_{11}+d_{22}\\
\end{array}
\right)
\quad
 ;
\quad
D(As_4^{2}) :
\left(\begin{array}{cccc}
d_{11}&0&0&0\\
0&d_{11}&0&0\\
d_{31}&d_{32}&2d_{11}&0\\
d_{41}&d_{42}&0&2d_{11}
\end{array}
\right)
$

$D(As_4^{3}) :
\left(\begin{array}{cccc}
d_{11}&0&0&0\\
0&d_{11}&0&0\\
d_{31}&d_{32}&2d_{11}&0\\
d_{41}&d_{42}&0&2d_{11}
\end{array}
\right)
\quad
 ;
\quad
D(As_4^{4}) :
\left(\begin{array}{cccc}
d_{11}&d_{12}&0&0\\
0&d_{22}&0&0\\
d_{31}&d_{32}&d_{11}+d_{22}&0\\
d_{41}&d_{42}&0&2d_{11}
\end{array}
\right)$

$D(As_4^{5}) :
\left(\begin{array}{cccc}
d_{11}&d_{12}&0&0\\
d_{21}&d_{22}&0&0\\
0&0&\frac{d_{11}}{2}+\frac{d_{22}}{2}&0\\
d_{41}&d_{42}&d_{43}&d_{11}+d_{22}
\end{array}
\right)
\quad
 ;
\quad
D(As_4^{6}) :
\left(\begin{array}{cccc}
d_{11}&d_{12}&0&0\\
0&d_{22}&0&0\\
d_{31}&d_{32}&2d_{22}&0\\
d_{41}&d_{42}&-\frac{2d_{12}}{\alpha-1}&d_{11}+d_{22}
\end{array}
\right)$

$D(As_4^{7}) :
\left(\begin{array}{cccc}
0&0&0&0\\
d_{21}&d_{22}&0&0\\
0&d_{32}&d_{33}&0\\
d_{32}&0&0&d_{33}-d_{22}
\end{array}
\right)
\quad
 ;
\quad
D(As_4^{8}) :
\left(\begin{array}{cccc}
0&0&0&0\\
d_{21}&d_{22}&d_{23}&d_{24}\\
-d_{24}&0&d_{33}&0\\
0&0&0&d_{22}-d_{33}
\end{array}
\right)$

$D(As_4^{9}) :
\left(\begin{array}{cccc}
0&0&0&0\\
d_{21}&d_{22}&d_{23}&d_{24}\\
-d_{24}&0&d_{33}&0\\
0&0&0&d_{22}-d_{33}
\end{array}
\right)
\quad
 ;
\quad
D(As_4^{10}) :
\left(\begin{array}{cccc}
0&0&0&0\\
0&0&0&0\\
0&0&0&0\\
d_{41}&d_{42}&0&0
\end{array}
\right)$

$D(As_4^{11}) :
\left(\begin{array}{cccc}
d_{11}&0&0&0\\
d_{21}&2d_{22}&0&0\\
d_{31}&d_{32}&3d_{11}&d_{34}\\
\frac{d_{21}-d_{34}}{2}&0&0&2d_{11}\\
\end{array}
\right)
\quad
 ;
\quad
D(As_4^{12}) :
\left(\begin{array}{cccc}
0&0&0&0\\
d_{21}&d_{22}&0&0\\
d_{31}&0&d_{33}&d_{34}\\
d_{41}&0&d_{43}&d_{44}
\end{array}
\right)$

$D(As_4^{13}) :
\left(\begin{array}{cccc}
d_{11}&d_{12}&0&0\\
0&0&0&0\\
0&d_{32}&d_{33}&d_{34}\\
0&d_{42}&d_{43}&d_{44}
\end{array}
\right)
\quad
 ;
\quad
D(As_4^{14}) :
\left(\begin{array}{cccc}
0&0&0&0\\
d_{21}&d_{22}&d_{23}&d_{24}\\
d_{31}&d_{32}&d_{33}&d_{34}\\
d_{41}&d_{42}&d_{43}&d_{44}
\end{array}
\right)$

$D(As_4^{15}) :
\left(\begin{array}{cccc}
d_{11}&d_{12}&d_{31}&d_{41}\\
0&0&0&0\\
d_{31}&d_{32}&d_{33}&d_{34}\\
d_{41}&d_{42}&d_{43}&d_{44}
\end{array}
\right)
\quad
 ;
\quad
D(As_4^{16}) :
\left(\begin{array}{cccc}
d_{11}&0&0&0\\
\frac{-2d_{11}+d_{44}}{1+\alpha}&d_{44}-d_{11}&0&0\\
0&0&\frac{d_{44}}{2}&0\\
d_{41}&d_{42}&d_{43}&d_{44}
\end{array}
\right)$

$D(As_4^{17}) :
\left(\begin{array}{cccc}
d_{11}&0&0&0\\
0&d_{11}&0&0\\
d_{31}&d_{32}&2d_{11}&0\\
d_{41}&d_{42}&0&2d_{11}
\end{array}
\right)
\quad
 ;
\quad
D(As_4^{18}) :
\left(\begin{array}{cccc}
d_{11}&0&0&0\\
0&d_{11}&0&0\\
d_{31}&d_{32}&2d_{11}&0\\
d_{41}&d_{42}&0&2d_{11}
\end{array}
\right)$

$D(As_4^{19}) :
\left(\begin{array}{cccc}
0&0&0&0\\
0&0&0&0\\
d_{31}&-d_{31}&d_{33}&0\\
d_{41}&0&0&d_{44}
\end{array}
\right)
\quad
 ;
\quad
D(As_4^{20}) :
\left(\begin{array}{cccc}
0&0&0&0\\
0&0&0&0\\
d_{31}&-d_{31}&d_{33}&0\\
d_{41}&0&0&d_{44}
\end{array}
\right)$

$D(As_4^{21}) :
\left(\begin{array}{cccc}
0&0&0&0\\
0&0&0&0\\
0&d_{32}&d_{33}&0\\
d_{41}&-d_{41}&0&d_{44}\\
\end{array}
\right)
\quad
 ;
\quad
D(As_4^{22}) :
\left(\begin{array}{cccc}
d_{11}&d_{12}&-d_{12}&0\\
0&0&0&0\\
0&0&0&0\\
0&d_{42}&0&d_{44}\\
\end{array}
\right)$

$D(As_4^{23}) :
\left(\begin{array}{cccc}
0&0&0&0\\
d_{21}&0&0&0\\
d_{31}&d_{32}&0&d_{34}\\
\frac{d_{21}-d_{34}}{2}&-d_{21}&0&0
\end{array}
\right)
\quad
 ;
\quad
D(As_4^{24}) :
\left(\begin{array}{cccc}
d_{11}&0&0&0\\
0&0&0&0\\
0&0&d_{33}&0\\
0&0&0&0
\end{array}
\right)$

$D(As_4^{25}) :
\left(\begin{array}{cccc}
d_{11}&0&0&0\\
d_{21}&d_{22}&0&0\\
0&0&d_{22}&0\\
d_{41}&d_{42}&d_{43}&d_{11}+d_{22}
\end{array}
\right)
\quad
 ;
\quad
D(As_4^{26}) :
\left(\begin{array}{cccc}
0&0&0&0\\
0&d_{22}&0&0\\
d_{31}&0&d_{33}&d_{34}\\
d_{41}&0&d_{43}&d_{44}
\end{array}
\right)$

$D(As_4^{27}) :
\left(\begin{array}{cccc}
0&0&0&0\\
0&d_{22}&0&0\\
d_{31}&0&d_{33}&d_{34}\\
d_{41}&0&d_{43}&d_{44}
\end{array}
\right)
\quad
 ;
\quad
D(As_4^{28}) :
\left(\begin{array}{cccc}
\frac{d_{44}}{2}&d_{12}&0&0\\
-d_{12}&\frac{d_{44}}{2}&0&0\\
0&0&\frac{d_{44}}{2}&0\\
d_{41}&d_{42}&d_{43}&d_{44}
\end{array}
\right)$

$D(As_4^{29}) :
\left(\begin{array}{cccc}
0&0&0&0\\
0&0&0&0\\
0&0&d_{33}&0\\
d_{41}&-d_{41}&0&d_{44}
\end{array}
\right)
\quad
 ;
\quad
D(As_4^{30}) :
\left(\begin{array}{cccc}
0&0&0&0\\
0&0&0&0\\
0&0&d_{33}&0\\
d_{41}&-d_{41}&0&d_{44}
\end{array}
\right)$

$D(As_4^{31}) :
\left(\begin{array}{cccc}
0&0&0&0\\
0&0&0&0\\
d_{31}&-d_{31}&d_{33}&0\\
d_{41}&-d_{41}&0&d_{44}
\end{array}
\right)
\quad
 ;
\quad
D(As_4^{32}) :
\left(\begin{array}{cccc}
d_{11}&0&0&0\\
0&0&0&0\\
0&0&d_{11}&0\\
0&0&0&0
\end{array}
\right)$

$D(As_4^{33}) :
\left(\begin{array}{cccc}
0&0&0&0\\
0&0&0&0\\
0&0&0&0\\
0&0&0&0\\
\end{array}
\right)
\quad
 ;
\quad
D(As_4^{34}) :
\left(\begin{array}{cccc}
d_{11}&d_{12}&d_{14}&0\\
0&0&0&0\\
0&-d_{14}&d_{33}&0\\
0&0&0&d_{11}-d_{33}
\end{array}
\right)$

$D(As_4^{35}) :
\left(\begin{array}{cccc}
d_{11}&d_{12}&d_{13}&0\\
0&0&0&0\\
0&0&d_{33}&0\\
0&0&0&d_{11}-d_{33}
\end{array}
\right)
\quad
 ;
\quad
D(As_4^{36}) :
\left(\begin{array}{cccc}
0&0&0&0\\
0&0&0&0\\
0&0&d_{33}&0\\
0&0&0&0
\end{array}
\right)$

$D(As_4^{37}) :
\left(\begin{array}{cccc}
0&0&0&0\\
0&d_{22}&0&0\\
0&0&0&0\\
0&0&0&d_{44}\\
\end{array}
\right)
\quad
 ;
\quad
D(As_4^{38}) :
\left(\begin{array}{cccc}
0&0&0&0\\
0&d_{22}&d_{23}&0\\
0&d_{32}&d_{33}&0\\
d_{41}&0&0&d_{44}
\end{array}
\right)$

$D(As_4^{39}) :
\left(\begin{array}{cccc}
0&0&0&0\\
0&d_{22}&0&0\\
0&d_{23}&2d_{22}&0\\
d_{41}&0&0&d_{44}
\end{array}
\right)
\quad
 ;
\quad
D(As_4^{40}) :
\left(\begin{array}{cccc}
0&0&0&0\\
0&d_{22}&d_{23}&d_{24}\\
0&d_{33}&d_{33}&d_{34}\\
0&d_{42}&d_{43}&d_{44}
\end{array}
\right)$

$D(As_4^{41}) :
\left(\begin{array}{cccc}
0&0&0&0\\
0&d_{22}&0&0\\
0&d_{32}&2d_{22}&0\\
d_{41}&0&0&d_{44}
\end{array}
\right)
\quad
 ;
\quad
D(As_4^{42}) :
\left(\begin{array}{cccc}
0&0&0&0\\
0&0&0&0\\
0&0&0&0\\
0&0&0&0\\
\end{array}
\right)$

$D(As_4^{43}) :
\left(\begin{array}{cccc}
0&0&0&0\\
0&d_{22}&0&0\\
0&d_{32}&d_{33}&0\\
0&d_{42}&d_{43}&2d_{22}
\end{array}
\right)
\quad
 ;
\quad
D(As_4^{44}) :
\left(\begin{array}{cccc}
0&0&0&0\\
0&d_{22}&0&0\\
0&0&d_{33}&0\\
0&d_{42}&d_{43}&d_{22}+d_{33}
\end{array}
\right)$

$D(As_4^{45}) :
\left(\begin{array}{cccc}
0&0&0&0\\
0&d_{22}&0&0\\
0&d_{32}&2d_{22}&0\\
0&d_{42}&2d_{32}&3d_{22}\\
\end{array}
\right)
\quad
 ;
\quad
D(As_4^{46}) :
\left(\begin{array}{cccc}
0&0&0&0\\
0&d_{22}&0&0\\
0&d_{32}&d_{22}&0\\
0&d_{42}&d_{43}&2d_{22}\\
\end{array}
\right)$
\end{thm}
\begin{proof}

Let \( \{e_1, e_2, e_3, e_4\} \) be a basis of the four-dimensional associative algebra \( As_4^2 \), where the multiplication rules are given by \( e_1e_2 = e_4 \) and \( e_3e_1 = e_4 \). Using the definition of a derivation, as in Definition 2.2, we have:

\[
D(e_1 \cdot e_2) = D(e_4) = 2d_{11}e_4, \quad D(e_1) \cdot e_2 + e_1 \cdot D(e_2)=2d_{11}.
\]

Next, applying the derivation property:

\[
D(e_3 \cdot e_1) = D(e_4) = 2d_{11}e_4, \quad D(e_3) \cdot e_1 + e_3 \cdot D(e_1)=2d_{11}.
\]

Thus, it follows that \( D(x \cdot y) = D(x) \cdot y + x \cdot D(y) \) for all \( x, y \in \{e_1, e_2, e_3, e_4\} \).

Hence, the matrix representation of the derivation \( D \) is given by:

\[
D = \begin{pmatrix}
d_{11} & 0 & 0 & 0 \\
0 & d_{11} & 0 & 0 \\
d_{31} & d_{32} & 2d_{11} & 0 \\
d_{41} & d_{42} & 0 & 2d_{11}
\end{pmatrix}
\quad \text{where} \quad \det(D) \neq 0.
\]
The derivation of the basis elements can be written as:
	\begin{align*}
		D(e_1)&=d_{11}e_1 + d_{31}e_3 + d_{41}e_4, \\
		D(e_2)&= d_{11}e_2 + d_{32}e_3 + d_{42}e_4, \\
		D(e_3)&=d_{31}e_1 + d_{32}e_2 + 2d_{11}e_3,\\
        D(e_4)&=d_{41}e_1 + d_{42}e_2 + 2d_{11}e_4.
	\end{align*}
Since for the four-dimensional algebra \( As_4^2 \), we have the relations \( e_1e_2 = e_4 \) and \( e_3e_1 = e_4 \), we need to consider the algebra structure and corresponding derivations for all possible combinations. After applying the algebra rules, we obtain the following matrix representation for the derivation:

\[
D(As_4^2) = \begin{pmatrix}
d_{11} & 0 & 0 & 0 \\
0 & d_{11} & 0 & 0 \\
d_{31} & d_{32} & 2d_{11} & 0 \\
d_{41} & d_{42} & 0 & 2d_{11}
\end{pmatrix}.
\]
Using similar methods for all algebras in Theorem 3.1, we obtain the derivations of all four-dimensional algebras as shown above.
\end{proof}

\begin{cor}\,
\begin{itemize}
	\item The dimensions of the derivation of $2$-dimensional associative algebras range between $0$ and $2$.
	\item The dimensions of the derivation of $3$-dimensional associative algebras range between $2$ and $4$.
	\item The dimensions of the derivation of $4$-dimensional associative algebras range between $0$ and $12$.
	\end{itemize}
\end{cor}

\section{ An Algorithmic Approach to Automorphism Groups}
 
Let \( A \) be an \( n \)-dimensional associative algebra over the field \( \mathbb{C} \), with the basis \( \{e_1, e_2, \dots, e_n\} \) and structure constants \( \gamma_{ij}^k \) such that:
\[
e_i \cdot e_j = \sum_{k=1}^n \gamma_{ij}^k e_k \quad \text{for} \quad i, j = 1, 2, \dots, n.
\]

The automorphism \( f \) can be expressed in terms of its action on the basis \( \{e_1, e_2, \dots, e_n\} \):
\[
f(e_i) = \sum_{j=1}^n a_{ji} e_j \quad \text{for} \quad i = 1, 2, \dots, n,
\]
where \( a_{ji} \) are complex coefficients. This means that the map \( f \) can be represented by the matrix \( A = (a_{ji}) \), with the components \( a_{ji} \) being the matrix entries of \( f \).

Equating the two expressions for \( f(e_i \cdot e_j) \) and \( f(e_i) \cdot f(e_j) \), we get the system of equations:

\[
\sum_{k=1}^n \gamma_{ij}^k a_{lk} = \sum_{p=1}^n \sum_{q=1}^n a_{pi} a_{qj} \gamma_{pq}^k, \quad \text{for each} \ i, j, k = 1, 2, \dots, n.
\]
These are \( n^3 \) equations for the \( n^2 \) unknowns \( a_{ji} \).

The solution set will form the automorphism group \( \text{Aut}(A) \).

\begin{thm}\label{thm4}
The \textbf{automorphism groups} of $2$-dimensional associative algebras have the
following form:
$Aut(As_2^1) =\left\{\begin{array}{ll}
\left(\begin{array}{cccc}
a_{11}&0\\
a_{21}&a_{11}^2\\
\end{array}
\right)
\end{array},\,
a_{11} \neq 0
\right\}
$
;

$
Aut(As_2^2) =\left\{
\begin{array}{ll}
\left(\begin{array}{cccc}
1&0\\
0&a_{22}\\
\end{array}
\right)
\end{array},\,
a_{22}\neq 0
\right\}
\quad ;
\\
Aut(As_2^3)=\left\{\begin{array}{ll}
\left(\begin{array}{cccc}
1&0\\
a_{21}&a_{22}\\
\end{array}
\right)
\end{array},\,
a_{22}\neq 0
\right\}$ ;

$
Aut(As_2^4)=\left\{\begin{array}{ll}
\left(\begin{array}{cccc}
1&0\\
a_{21}&a_{22}\\
\end{array}
\right)
\end{array}, \,
a_{22}\neq 0\right\}
$
;

$
Aut(As_2^5)=\left\{\begin{array}{ll}
\left(\begin{array}{cccc}
1&0\\
0&1
\end{array}
\right)
\end{array},
\left(\begin{array}{cccc}
1&2\alpha\\
0&-1\\
\end{array}
\right)
\right\}
$.
\end{thm}
\begin{proof}
Let $\{e_1, e_2\}$ be a basis of two-dimensional associative algebra, $As_2$ and\\
$\psi=
\left(\begin{array}{cccc}
a_{11}&a_{12}\\
a_{21}&a_{22}\\
\end{array}
\right)$
be a nonsingular matrix where $\psi \in Aut(As_2)$. Suppose $\{e^{\prime}_1, e^{\prime}_2\}$ be a
new basis that obtain by simply multiplying $\psi$ with the basis i.e.,
\begin{align*}
\left(\begin{array}{c}
e^{\prime}_1\\
e^{\prime}_2\\
\end{array}
\right)&=\left(\begin{array}{cccc}
a_{11}&a_{12}\\
a_{21}&a_{22}\\
\end{array}
\right)^{T}
\left(\begin{array}{c}
e_{1}\\
e_{2}\\
\end{array}
\right)
\end{align*}
Thus, $\{e^{\prime}_1,e^{\prime}_2\}$ can be written as follows

\begin{equation}\label{eq22}
e^{\prime}_1=a_{11}e_1+a_{21}e_2 \quad and \quad e^{\prime}_2=a_{12}e_1+ a_{22}e_2.
\end{equation}
Now consider the algebra $As^1_2: e_1e_1 = e_2$. By applying the new basis (\ref{eq22}) to the algebra we get the following table of multiplications:
\begin{align*}
e^{\prime}_1e^{\prime}_1&=a^2_{11}e_2=a_{12}e_1+ a_{22}e_2, \quad e^{\prime}_1e^{\prime}_2=a_{11}a_{12}e_2=0,\\
e^{\prime}_2e^{\prime}_1&=a_{11}a_{12}e_2=0, \quad \quad  \quad \quad \quad e^{\prime}_2e^{\prime}_2=a^2_{12}e_2=0.
\end{align*}

Then we have the system $a^2_{11} = a_{22}, \; a_{11}a_{12} = 0, \; a_{12} = 0$ and $a_{21}$ is any. By solving
the system, we obtain the group of automorphism for $As^1_2$ as follows
\begin{align*}
 Aut(As^1_2)&= \left(\begin{array}{cccc}
a_{11}&0\\
a_{21}&a^2_{11}\\
\end{array}
\right)
,  where \quad a_{11}\neq 0.
\end{align*}
Since $\{a_{11}, a_{21}\}$ is a basis of $Aut(As^1_2)$, therefore $dim(Aut(As^1_2)) = 1$.
By applying the similar method for other algebras as shown above, we get all the list of automorphism groups as in Theorem \ref{thm4}.
\end{proof}
\begin{thm}\label{thm5}
The \textbf{automorphism groups} of $3$-dimensional associative algebras have the
following form:

$
Aut(As_3^1)=
\left\{\begin{array}{ll}
\left(\begin{array}{cccc}
a_{11}&0&0\\
a_{21}&a_{11}a_{33}&a_{23}\\
0&0&a_{33}
\end{array}
\right)
\end{array},
\begin{array}{ll}
a_{11} \neq 0\\
or\\
a_{33} \neq 0
\end{array}\right\}
 ;
Aut(As_3^2)=
\left\{\begin{array}{ll}
\left(\begin{array}{cccc}
a_{11}&0&0\\
a_{21}&a_{11}a_{33}&a_{23}\\
0&0&a_{33}
\end{array}
\right)
\end{array},
\begin{array}{ll}
a_{11} \neq 0\\
or\\
a_{33} \neq 0
\end{array}
\right\}
$

$
Aut(As_3^3)=
\left\{\begin{array}{ll}
\left(\begin{array}{cccc}
a_{11}&0&0\\
a_{21}&a_{11}^2&0\\
a_{31}&2a_{11}a_{21}&a_{11}^3\\
\end{array}
\right)
\end{array}
,
\begin{array}{ll}
a_{11} \neq 0
\end{array}\right\}
;
Aut(As_3^4)=
\left\{\begin{array}{ll}
\left(\begin{array}{cccc}
a_{11}&0&0\\
a_{22}-a_{11}&a_{22}&a_{23}\\
0&0&1
\end{array}
\right)
\end{array},
\begin{array}{ll}
a_{11} \neq 0\\
or\\
a_{22} \neq 0
\end{array}\right\}
$

$
Aut(As_3^5)=
\left\{\begin{array}{ll}
\left(\begin{array}{cccc}
a_{11}&0&a_{13}\\
0&a_{22}&a_{23}\\
0&0&1
\end{array}
\right)
\end{array},
\begin{array}{ll}
a_{11} \neq 0\\
or\\
a_{22} \neq 0
\end{array}\right\}
 ;
Aut(As_3^6)=
\left\{\begin{array}{ll}
\left(\begin{array}{cccc}
a_{11}&0&0\\
a_{22}-a_{11}&a_{22}&a_{23}\\
0&0&1
\end{array}
\right)
\end{array},
\begin{array}{ll}
a_{11} \neq 0\\
or\\
a_{22} \neq 0
\end{array}\right\}
$

$
Aut(As_3^7)=
\left\{
\begin{array}{ll}
\left(\begin{array}{cccc}
a_{11}&a_{12}&-a_{12}\\
0&1&0\\
0&0&1
\end{array}
\right)
\end{array},
\begin{array}{ll}
a_{11} \neq 0
\end{array}
\right\}
 ;
Aut(As_3^8)=
\left\{
\begin{array}{ll}
\left(\begin{array}{cccc}
a_{11}&0&0\\
0&a_{22}&a_{23}\\
0&0&1
\end{array}
\right)
\end{array},
\begin{array}{ll}
a_{11} \neq 0\\
or\\
a_{22} \neq 0
\end{array}
\right\}
$

$
Aut(As_3^9)=
\left\{
\begin{array}{ll}
\left(\begin{array}{cccc}
a_{11}&0&a_{13}\\
0&a_{22}&0\\
0&0&1
\end{array}
\right)
\end{array},
\begin{array}{ll}
a_{11} \neq 0\\
or\\
a_{33} \neq 0
\end{array}
\right\}
;
Aut(As_3^{10})=
\left\{
\begin{array}{ll}
\left(\begin{array}{cccc}
a_{11}&a_{12}&0\\
a_{21}&a_{22}&0\\
0&0&1
\end{array}
\right)
\end{array},
\begin{array}{ll}
a_{11}a_{22} \neq a_{12}a_{21}
\end{array}
\right\}
$

$
Aut(As_3^{11})=
\left\{
\begin{array}{ll}
\left(\begin{array}{cccc}
a_{11}&0&0\\
a_{22}-a_{11}&a_{22}&0\\
0&0&1
\end{array}
\right)
\end{array},
\begin{array}{ll}
a_{11} \neq 0\\
or\\
a_{22} \neq 0
\end{array}
\right\}
;
Aut(As_3^{12})=
\left\{
\begin{array}{ll}
\left(\begin{array}{cccc}
a_{11}&0&0\\
a_{21}&a_{11}^2&0\\
0&0&1
\end{array}
\right)
\end{array},
\begin{array}{ll}
a_{11} \neq 0
\end{array}
\right\}
$
\end{thm}

\begin{proof}
Let $\{e_1, e_2,e_3\}$ be a basis of three-dimensional associative algebra, $As_3$ and\\
$\psi=
\left(\begin{array}{cccc}
a_{11}&a_{12}&a_{13}\\
a_{21}&a_{22}&a_{23}\\
a_{31}&a_{32}&a_{33}\\
\end{array}
\right)$
be a nonsingular matrix where $\psi\in Aut(As_3)$. Suppose $\{e^{\prime}_1, e^{\prime}_2, e^{\prime}_3\}$ be a
new basis that obtain by simply multiplying $\psi$ with the basis i.e.,
\begin{align*}
\left(\begin{array}{c}
e^{\prime}_1\\
e^{\prime}_2\\
e^{\prime}_3\\
\end{array}
\right)&=\left(\begin{array}{cccc}
a_{11}&a_{12}&a_{13}\\
a_{21}&a_{22}&a_{23}\\
a_{31}&a_{32}&a_{33}\\
\end{array}
\right)^{T}
\left(\begin{array}{c}
e_{1}\\
e_{2}\\
e_{3}\\
\end{array}
\right)
\end{align*}
Thus, $\{e^{\prime}_1,e^{\prime}_2,e^{\prime}_3\}$ can be written as follows

\begin{equation}
e^{\prime}_1=a_{11}e_1+a_{21}e_2+a_{31}e_3,e^{\prime}_2=a_{12}e_1+a_{22}e_2+a_{32}e_3  \;and\; e^{\prime}_3=a_{13}e_1+ a_{23}e_2+ a_{33}e_3.
\end{equation}
Now consider the algebra $As^8_3: e_1 e_3 = e_1, e_2 e_3 = e_2 , e_3 e_1 = e_1, e_3 e_3 = e_3$. By applying the new basis
to the algebra we get the following table of multiplications:
\begin{align*}
e^{\prime}_1e^{\prime}_1&=2a_{11}a_{31}e_1+a_{21}a_{31}e_2+a_{31}^2e_3=0,\\
e^{\prime}_1e^{\prime}_2&=(a_{11}a_{33}+a_{12}a_{33})e_1+a_{21}a_{33}e_2+a_{31}a_{33}e_3=0,\\
e^{\prime}_1e^{\prime}_3&=(a_{11}a_{33}+a_{13}a_{31})e_1+a_{21}a_{33}e_2+a_{31}a_{33}e_3=a_{11}e_1+a_{21}e_2+a_{31}e_3,\\
e^{\prime}_2e^{\prime}_1&=(a_{12}a_{31}+a_{11}a_{22})e_1+a_{22}a_{32}e_2+a_{31}a_{32}e_3=0, \\
e^{\prime}_2e^{\prime}_2&=2a_{12}a_{32}e_1+a_{22}a_{32}e_2+a_{33}^2e_3=0,\\
e^{\prime}_2e^{\prime}_3&=(a_{12}a_{33}+a_{32}a_{13})e_1+a_{22}a_{33}e_2+a_{32}a_{33}e_3=a_{12}e_1+a_{22}e_2+a_{32}e_3,\\
e^{\prime}_3e^{\prime}_1&=(a_{13}a_{31}+a_{11}a_{33})e_1+a_{23}a_{31}e_2+a_{31}a_{33}e_3=a_{11}e_1+a_{21}e_2+a_{31}e_3,\\
e^{\prime}_3e^{\prime}_2&=(a_{13}a_{32}+a_{33}a_{12})e_1+a_{23}a_{32}e_2+a_{33}a_{32}e_3=0,\\
e^{\prime}_3e^{\prime}_3&=2a_{13}a_{33}e_1+a_{23}a_{33}e_2+a^2_{33}a_{33}e_3=a_{13}e_1+a_{23}e_2+a_{33}e_3.
\end{align*}

Then we have the system $a_{12} = a_{13}=a_{21} = a_{31}= a_{32}=0  , \; a_{33}= 1$ and $a_{11}, a_{22}, a_{23}$ are any. By solving
the system, we obtain the group of automorphism for $As^8_3$ as follows
\begin{align*}
 Aut(As^8_3)&= \left(\begin{array}{cccc}
a_{11}&0&0\\
0&a_{22}&a_{23}\\
0&0&1\\
\end{array}
\right)
, \text{where} \quad a_{11}\neq 0\;or\;a_{22}\neq 0.
\end{align*}
Since $\{a_{11}, a_{22}, a_{23}\}$ is a basis of $Aut(As^8_3)$, therefore $dim(Aut(As^8_3))=3$.
By applying the similar method for other algebras as shown
above, we get all the list of automorphism groups as in Theorem \ref{thm5}.
\end{proof}

\begin{thm}\label{thm6}
The \textbf{automorphism groups} of $4$-dimensional associative algebras have the
following form:

$Aut(As_4^{1})=
\left\{
\begin{array}{ll}
\left(\begin{array}{cccc}
a_{11}&0&0&0\\
0&a_{22}&0&0\\
a_{31}&a_{32}&a_{11}a_{22}&0\\
a_{41}&a_{42}&0&a_{11}a_{22}
\end{array}
\right)
\end{array},
\begin{array}{ll}
a_{11} or a_{22} \neq 0
\end{array}
\right\}
$
;

$Aut(As_4^{2})=
\left\{
\begin{array}{ll}
\left(\begin{array}{cccc}
a_{11}&0&0&0\\
a_{21}&a_{22}&0&0\\
-a_{21}&0&a_{22}&0\\
a_{41}&a_{42}&a_{43}&a_{11}a_{22}
\end{array}
\right)
\end{array},
\begin{array}{ll}
a_{11} or a_{22} \neq 0
\end{array}
\right\}
$

$
Aut(As_4^{3})=
\left\{
\begin{array}{ll}
\left(\begin{array}{cccc}
a_{22}&0&0&0\\
0&a_{22}&0&0\\
a_{31}&a_{32}&a_{22}^2&0\\
a_{41}&a_{42}&a_{43}&a_{22}^2
\end{array}
\right)
\end{array},
\begin{array}{ll}
a_{22} \neq 0
\end{array}
\right\}
$

$
Aut(As_4^{4})=
\left\{
\begin{array}{ll}
\left(\begin{array}{cccc}
a_{11}&a_{12}&0&0\\
0&a_{22}&0&0\\
a_{31}&a_{32}&a_{11}a_{22}&0\\
a_{41}&a_{42}&0&a_{22}^2
\end{array}
\right)
\end{array},
\begin{array}{ll}
a_{11}\, or \, a_{22} \neq 0\\
\end{array}
\right\}
$

$Aut(As_4^{5})=
\left\{
\begin{array}{ll}
\left(\begin{array}{cccc}
a_{11}&a_{12}&0&0\\
a_{21}&\frac{a_{12}a_{21}+a_{33}^2}{a_{11}}&0&0\\
0&0&a_{33}&0\\
a_{41}&a_{42}&a_{43}&a_{33}^2
\end{array}
\right)
\end{array}, a_{33} \neq 0;
\begin{array}{ll}
\left(\begin{array}{cccc}
0&a_{12}&0&0\\
\frac{a_{33}^2}{a_{21}}&a_{22}&0&0\\
0&0&a_{33}&0\\
a_{41}&a_{42}&a_{43}&a_{33}^2
\end{array}
\right),
\end{array}
\begin{array}{ll}
a_{12}\, or\,
a_{21}\, or\,
a_{33} \neq 0
\end{array}
\right\}
$

$
Aut(As_4^{6})=
\left\{
\begin{array}{ll}
\left(\begin{array}{cccc}
a_{11}&a_{12}&0&0\\
0&a_{22}&0&0\\
a_{31}&a_{32}&a_{22}^2&0\\
a_{41}&a_{42}&\frac{2a_{12}a_{22}}{\alpha-1}&a_{11}a_{22}
\end{array}
\right)
\end{array},
\begin{array}{ll}
a_{11}\, or\, a_{22}\, \neq 0\\
\end{array}
\right\}
$

$
Aut(As_4^{7})=
\left\{
\begin{array}{ll}
\left(\begin{array}{cccc}
1&0&0&0\\
a_{21}&a_{22}&0&0\\
a_{21}a_{41}&a_{22}a_{41}&a_{22}a_{44}&a_{21}a_{44}\\
a_{41}&0&0&a_{44}
\end{array}
\right)
\end{array},
\begin{array}{ll}
a_{22}\, or \, a_{44} \neq 0
\end{array}
\right\}
$

$
Aut(As_4^{8})=
\left\{
\begin{array}{ll}
\left(\begin{array}{cccc}
1&0&0&0\\
a_{21}&a_{22}&a_{23}&0\\
0&0&a_{33}&0\\
0&0&0&\frac{a_{22}}{a_{33}}
\end{array}
\right)
\end{array},
\begin{array}{ll}
a_{22}\, \neq 0
\end{array}
\right\}
$

$Aut(As_4^{9})=
\left\{
\begin{array}{ll}
\left(\begin{array}{cccc}
1&0&0&0\\
a_{21}&a_{22}&a_{23}&\frac{a_{22}a_{31}}{a_{33}}\\
a_{31}&0&a_{33}&0\\
0&0&0&\frac{a_{22}}{a_{33}}
\end{array}
\right)
\end{array};
\begin{array}{ll}
\left(\begin{array}{cccc}
1&0&0&0\\
a_{21}&a_{22}&a_{23}&0\\
0&0&a_{33}&0\\
0&0&0&\frac{a_{22}}{a_{33}}
\end{array}
\right),
\end{array}
\begin{array}{ll}
a_{22}\, \neq 0
\end{array}
\right\}
$

$
Aut(As_4^{10})=
\left\{
\begin{array}{ll}
\left(\begin{array}{cccc}
1&0&0&0\\
0&1&0&0\\
0&0&1&0\\
0&0&0&1
\end{array}
\right)
\end{array}
\right\}
$

$
Aut(As_4^{11})=
\left\{
\begin{array}{ll}
\left(\begin{array}{cccc}
a_{11}&0&0&0\\
a_{21}&a_{11}^2&0&0\\
a_{31}&a_{32}&a_{11}^3&a_{21}-2a_{41}\\
a_{41}&0&0&a_{11}^2
\end{array}
\right)
\end{array},
\begin{array}{ll}
a_{11} \neq 0
\end{array}
\right\}
$

$
Aut(As_4^{12})=
\left\{
\begin{array}{ll}
\left(\begin{array}{cccc}
1&0&0&0\\
a_{21}&a_{22}&0&0\\
a_{31}&0&a_{33}&a_{34}\\
a_{41}&0&a_{43}&a_{44}
\end{array}
\right)
\end{array},
\begin{array}{ll}
a_{22}(a_{33}a_{44}-a_{34}a_{43}) \neq 0
\end{array}
\right\}
$

$
Aut(As_4^{13})=
\left\{
\begin{array}{ll}
\left(\begin{array}{cccc}
a_{11}&a_{12}&0&0\\
1&0&0&0\\
0&a_{32}&a_{33}&a_{34}\\
0&a_{42}&a_{43}&a_{44}
\end{array}
\right)
\end{array},
\begin{array}{ll}
a_{12}(a_{34}a_{43}-a_{33}a_{44})\neq 0
\end{array}
\right\}
$

$
Aut(As_4^{14})=
\left\{
\begin{array}{ll}
\left(\begin{array}{cccc}
1&0&0&0\\
0&a_{22}&0&0\\
0&0&1&0\\
a_{41}&0&a_{41}&a_{44}
\end{array}
\right)
\end{array},
\begin{array}{ll}
a_{22} \,or \,
a_{44}\neq 0
\end{array}
\right\}
$

$
Aut(As_4^{15})=
\left\{
\begin{array}{ll}
\left(\begin{array}{cccc}
a_{11}&a_{12}&a_{13}&a_{14}\\
0&1&0&0\\
a_{31}&a_{32}&a_{33}&a_{34}\\
a_{41}&a_{42}&a_{43}&a_{44}
\end{array}
\right)
\end{array},
\begin{array}{ll}
a_{14}(a_{31}a_{43}-a_{33}a_{41})+\\
a_{11}(a_{33}a_{44}-a_{34}a_{43})+\\
a_{13}(a_{34}a_{41}-a_{31}a_{44})\neq 0
\end{array}
\right\}
$

$
Aut(As_4^{16})=
\left\{
\begin{array}{ll}
\left(\begin{array}{cccc}
a_{11}&0&0&0\\
a_{21}&a_{11}+a_{21}+\alpha a_{21}&0&0\\
0&0&\sqrt{a_{11}}\sqrt{a_{11}+a_{21}+\alpha a_{21}}&0\\
a_{41}&a_{42}&a_{43}&a_{11}(a_{11}+a_{21}+\alpha a_{21})
\end{array}
\right)
\end{array}
\right\}\\
\begin{array}{ll}
a_{11}^\frac{3}{2}(a_{11}+a_{21}+\alpha a_{21})^\frac{3}{2}
(a_{11}^2+a_{11}a_{21}+\alpha a_{11}a_{21})\neq 0
\end{array}
$

$
Aut(As_4^{17})=
\left\{
\begin{array}{ll}
\left(\begin{array}{cccc}
a_{22}&0&0&0\\
0&a_{22}&0&0\\
a_{31}&a_{32}&a_{22}^2&0\\
a_{41}&a_{42}&0&a_{22}^2
\end{array}
\right)
\end{array}
,
\begin{array}{ll}
a_{22}\neq 0
\end{array}
;
\begin{array}{ll}
\left(\begin{array}{cccc}
0&a_{21}&0&0\\
a_{21}&0&0&0\\
a_{31}&a_{32}&-a_{21}^2&-2a_{21}^2\\
a_{41}&a_{42}&0&a_{21}^2
\end{array}
\right)
\end{array},
\begin{array}{ll}
a_{21}\neq 0
\end{array}
\right\}
$

$
Aut(As_4^{18})=
\left\{
\begin{array}{ll}
\left(\begin{array}{cccc}
a_{11}&0&0&0\\
0&a_{11}&0&0\\
a_{31}&a_{32}&a_{11}^2&0\\
a_{41}&a_{42}&0&a_{11}^2
\end{array}
\right)
\end{array},
\begin{array}{ll}
a_{11}\neq 0
\end{array};
\begin{array}{ll}
\left(\begin{array}{cccc}
0&a_{12}&0&0\\
\frac{a_{21}}{\alpha}&0&0&0\\
a_{31}&a_{32}&0&-\frac{a_{12}^2}{\alpha^2}\\
a_{41}&a_{42}&-a_{12}^2&0
\end{array}
\right)
\end{array},
\begin{array}{ll}
\frac{a_{12}^5a_{21}}{\alpha^3}\neq 0
\end{array}
\right\}
$

$
Aut(As_4^{19})=
\left\{
\begin{array}{ll}
\left(\begin{array}{cccc}
1&0&0&0\\
0&1&0&0\\
0&0&a_{33}&0\\
a_{41}&0&0&a_{44}
\end{array}
\right)
\end{array};
\begin{array}{ll}
\left(\begin{array}{cccc}
1&0&0&0\\
0&1&0&0\\
0&-a_{31}&a_{33}&0\\
a_{41}&0&0&a_{44}
\end{array}
\right)
\end{array},
\begin{array}{ll}
a_{33}\, or\,
a_{44}\,\neq 0
\end{array}
\right\}
$

$
Aut(As_4^{20})=
\left\{
\begin{array}{ll}
\left(\begin{array}{cccc}
1&0&0&0\\
0&1&0&0\\
0&0&a_{33}&0\\
a_{41}&0&0&a_{44}
\end{array}
\right)
\end{array},
\begin{array}{ll}
a_{33}\, or\,
a_{44}\neq 0
\end{array}
\right\}
$

$
Aut(As_4^{21})=
\left\{
\begin{array}{ll}
\left(\begin{array}{cccc}
1&0&0&0\\
0&1&0&0\\
0&0&a_{33}&0\\
a_{41}&-a_{41}&0&a_{44}
\end{array}
\right)
\end{array};
\begin{array}{ll}
\left(\begin{array}{cccc}
1&0&0&0\\
0&1&0&0\\
0&0&a_{33}&0\\
0&0&0&a_{44}
\end{array}
\right)
\end{array}
,
\begin{array}{ll}
a_{33}\, or\,
a_{44}\neq 0
\end{array}
\right\}
$

$
Aut(As_4^{22})=
\left\{
\begin{array}{ll}
\left(\begin{array}{cccc}
a_{11}&0&0&0\\
0&1&0&0\\
0&0&1&0\\
0&a_{42}&0&a_{44}
\end{array}
\right)
\end{array};
\begin{array}{ll}
\left(\begin{array}{cccc}
a_{11}&a_{12}&-a_{12}&0\\
0&1&0&0\\
0&0&1&0\\
0&a_{42}&0&a_{44}
\end{array}
\right)
\end{array}
,
\begin{array}{ll}
a_{11}\, or \,
a_{44}\neq 0
\end{array}
\right\}
$

$
Aut(As_4^{23})=
\left\{
\begin{array}{ll}
\left(\begin{array}{cccc}
1&0&0&0\\
a_{21}&1&0&0\\
a_{31}&a_{32}&1&a_{21}+a_{21}^2-2a_{41}\\
a_{41}&a_{21}&0&1
\end{array}
\right)
\end{array}
\begin{array}{ll}
\end{array}
\right\}
$

$
Aut(As_4^{24})=
\left\{
\begin{array}{ll}
\left(\begin{array}{cccc}
1&0&0&0\\
0&1&0&0\\
0&a_{32}&a_{33}&0\\
0&a_{42}&0&a_{44}
\end{array}
\right)
\end{array}
,
\begin{array}{ll}
a_{33} \,or \,
a_{44}\neq 0
\end{array}
\right\}
$

$
Aut(As_4^{25})=
\left\{
\begin{array}{ll}
\left(\begin{array}{cccc}
a_{22}&0&0&0\\
0&a_{22}&0&0\\
0&0&a_{22}&0\\
a_{41}&a_{42}&a_{43}&a_{22}^2
\end{array}
\right)
\end{array};
\begin{array}{ll}
\left(\begin{array}{cccc}
a_{22}&0&0&0\\
a_{21}&a_{22}&0&0\\
\frac{-a_{21}^2}{2a_{22}}&-a_{21}&a_{22}&0\\
a_{41}&a_{42}&a_{43}&a_{22}^2
\end{array}
\right)
\end{array}
,
\begin{array}{ll}
a_{22}\neq 0
\end{array}
\right\}
$

$
Aut(As_4^{26})=
\left\{
\begin{array}{ll}
\left(\begin{array}{cccc}
1&0&0&0\\
0&a_{22}&0&0\\
a_{31}&0&a_{33}&a_{34}\\
a_{41}&0&a_{43}&a_{44}
\end{array}
\right)
\end{array}
,
\begin{array}{ll}
a_{22}(a_{33}a_{44}-a_{34}a_{43})\neq 0
\end{array}
\right\}
$

$
Aut(As_4^{27})=
\left\{
\begin{array}{ll}
\left(\begin{array}{cccc}
1&0&0&0\\
0&a_{22}&0&0\\
a_{31}&0&a_{33}&a_{34}\\
a_{41}&0&a_{43}&a_{44}
\end{array}
\right)
\end{array}
,
\begin{array}{ll}
a_{22}(a_{33}a_{44}-a_{34}a_{43})\neq 0
\end{array}
\right\}
$

$
Aut(As_4^{28})=
\left\{
\begin{array}{ll}
\left(\begin{array}{cccc}
1&0&0&0\\
0&1&0&0\\
0&0&1&0\\
a_{41}&0&a_{43}&1
\end{array}
\right)
\end{array}
\right\}
$

$
Aut(As_4^{29})=
\left\{
\begin{array}{ll}
\left(\begin{array}{cccc}
1&0&0&0\\
0&1&0&0\\
0&0&a_{33}&0\\
a_{41}&-a_{41}&0&a_{44}
\end{array}
\right)
\end{array}
;
\begin{array}{ll}
\left(\begin{array}{cccc}
1&0&0&0\\
0&1&0&0\\
0&0&a_{33}&0\\
0&0&a_{43}&a_{44}
\end{array}
\right)
\end{array}
,
\begin{array}{ll}
a_{33}\, or \,
a_{44}\neq 0
\end{array}
\right\}
$

$
Aut(As_4^{30})=
\left\{
\begin{array}{ll}
\left(\begin{array}{cccc}
1&0&0&0\\
0&1&0&0\\
0&0&a_{33}&0\\
0&0&a_{43}&a_{44}
\end{array}
\right)
\end{array}
,
\begin{array}{ll}
a_{33}\; or \;
a_{44}\neq 0
\end{array}
\right\}
$

$
Aut(As_4^{31})=
\left\{
\begin{array}{ll}
\left(\begin{array}{cccc}
1&0&0&0\\
0&1&0&0\\
a_{31}&-a_{31}&a_{33}&0\\
a_{41}&-a_{41}&0&a_{44}
\end{array}
\right)
\end{array}
,
\begin{array}{ll}
a_{33}\; or \;
a_{44}\neq 0
\end{array}
\right\}
$

$
Aut(As_4^{32})=
\left\{
\begin{array}{ll}
\left(\begin{array}{cccc}
a_{11}&0&0&0\\
0&\frac{a_{33}}{a_{11}}&0&0\\
0&0&a_{33}&0\\
0&0&0&1
\end{array}
\right)
\end{array}
,
\begin{array}{ll}
a_{33}\neq 0
\end{array}
\right\}
$

$
Aut(As_4^{33})=
\left\{
\begin{array}{ll}
\left(\begin{array}{cccc}
a_{11}&0&0&0\\
a_{21}&a_{11}^2&0&0\\
a_{31}&2a_{11}a_{21}&a_{11}^3&0\\
a_{41}&a_{21}^2+2a_{11}a_{31}&3a_{11}^2a_{21}&a_{11}^4
\end{array}
\right)
\end{array}
,
\begin{array}{ll}
a_{11}\neq 0
\end{array}
\right\}
$

$
Aut(As_4^{34})=
\left\{
\begin{array}{ll}
\left(\begin{array}{cccc}
a_{33}a_{44}&a_{12}&a_{13}&a_{32}a_{44}\\
0&1&0&0\\
0&a_{32}&a_{33}&0\\
0&0&0&a_{44}
\end{array}
\right)
\end{array}
,
\begin{array}{ll}
a_{33}\;or\;a_{44}\neq 0
\end{array}
\right\}
$

$
Aut(As_4^{35})=
\left\{
\begin{array}{ll}
\left(\begin{array}{cccc}
a_{33}a_{44}&a_{12}&a_{13}&a_{32}a_{44}\\
0&1&0&0\\
0&a_{32}&a_{33}&0\\
0&0&0&a_{44}
\end{array}
\right)
\end{array}
,
\begin{array}{ll}
a_{33}\;or\;a_{44}\neq 0
\end{array}
\right\}
$

$
Aut(As_4^{36})=
\left\{
\begin{array}{ll}
\left(\begin{array}{cccc}
1&0&0&0\\
0&1&0&0\\
a_{31}&-a_{31}&a_{33}&a_{34}\\
a_{41}&-a_{41}&a_{43}&a_{44}
\end{array}
\right)
\end{array}
;
\begin{array}{ll}
\left(\begin{array}{cccc}
1&0&0&0\\
0&1&0&0\\
0&0&a_{33}&a_{34}\\
0&0&a_{43}&a_{44}
\end{array}
\right)
\end{array}
,
\begin{array}{ll}
a_{33}a_{44}\neq a_{34}a_{43}
\end{array}
\right\}
$

$
Aut(As_4^{38})=
\left\{
\begin{array}{ll}
\left(\begin{array}{cccc}
1&0&0&0\\
0&a_{22}&a_{23}&0\\
0&a_{32}&a_{33}&0\\
a_{41}&0&0&a_{44}
\end{array}
\right)
\end{array}
,
\begin{array}{ll}
a_{44}(a_{22}a_{33}-a_{23}a_{32})\neq 0
\end{array}
\right\}
$

$
Aut(As_4^{39})=
\left\{
\begin{array}{ll}
\left(\begin{array}{cccc}
1&0&0&0\\
0&a_{22}&0&0\\
0&a_{32}&a_{22}^2&0\\
a_{41}&0&0&a_{44}
\end{array}
\right)
\end{array}
,
\begin{array}{ll}
a_{22} \; or \;
a_{44}\neq 0
\end{array}
\right\}
$

$
Aut(As_4^{40})=
\left\{
\begin{array}{ll}
\left(\begin{array}{cccc}
1&0&0&0\\
0&a_{22}&a_{23}&a_{24}\\
0&a_{32}&a_{22}^2&a_{34}\\
0&a_{42}&a_{43}&a_{44}
\end{array}
\right)
\end{array}
,
\begin{array}{ll}
a_{23}a_{34}a_{42}-a_{22}^2a_{24}a_{42}+\\
a_{24}a_{32}a_{43}-a_{22}a_{34}a_{43}+\\
a_{22}^3a_{44}-a_{23}a_{32}a_{44})\neq 0
\end{array}
\right\}
$

$
Aut(As_4^{41})=
\left\{
\begin{array}{ll}
\left(\begin{array}{cccc}
1&0&0&0\\
0&a_{22}&0&0\\
0&a_{32}&a_{22}^2&0\\
a_{41}&0&0&a_{44}
\end{array}
\right)
\end{array}
,
\begin{array}{ll}
a_{22} \, or \,
a_{44}\neq 0
\end{array}
\right\}
$

$
Aut(As_4^{42})=
\left\{
\begin{array}{ll}
\left(\begin{array}{cccc}
1&0&0&0\\
0&a_{22}&0&0\\
0&0&\frac{1}{a_{22}}&0\\
0&0&0&1
\end{array}
\right)
\end{array}
,
\begin{array}{ll}
a_{22} \,\neq 0
\end{array}
\right\}
,
$

$
Aut(As_4^{43})=
\left\{
\begin{array}{ll}
\left(\begin{array}{cccc}
1&0&0&0\\
0&a_{22}&0&0\\
0&a_{32}&a_{33}&0\\
0&a_{42}&a_{43}&a_{22}^2
\end{array}
\right)
\end{array}
,
\begin{array}{ll}
a_{22} \; or \;
a_{33}\neq 0
\end{array}
\right\}
$

$
Aut(As_4^{44})=
\left\{
\begin{array}{ll}
\left(\begin{array}{cccc}
1&0&0&0\\
0&a_{22}&0&0\\
0&0&a_{33}&0\\
0&a_{42}&a_{43}&a_{22}a_{33}
\end{array}
\right)
\end{array}
,
\begin{array}{ll}
a_{22} \; or \;
a_{33}\neq 0
\end{array}
\right\}
$

$
Aut(As_4^{45})=
\left\{
\begin{array}{ll}
\left(\begin{array}{cccc}
1&0&0&0\\
0&a_{22}&0&0\\
0&a_{32}&a_{22}^2&0\\
0&a_{42}&2a_{22}a_{33}&a_{22}^3
\end{array}
\right)
\end{array}
,
\begin{array}{ll}
a_{22}\neq 0
\end{array}
\right\}
$

$
Aut(As_4^{46})=
\left\{
\begin{array}{ll}
\left(\begin{array}{cccc}
1&0&0&0\\
0&a_{22}&0&0\\
0&a_{32}&a_{22}^2&0\\
0&a_{42}&a_{43}&a_{22}^2
\end{array}
\right)
\end{array}
,
\begin{array}{ll}
a_{22}\neq 0
\end{array}
\right\}
$.
\end{thm}

\begin{proof}
Let $\{e_1, e_2,e_3,e_4\}$ be a basis of four-dimensional associative algebra, $As_4$ and\\
$\psi=
\left(\begin{array}{cccc}
a_{11}&a_{12}&a_{13}&a_{14}\\
a_{21}&a_{22}&a_{23}&a_{24}\\
a_{31}&a_{32}&a_{33}&a_{34}\\
a_{41}&a_{42}&a_{43}&a_{44}\\
\end{array}
\right)$
be a nonsingular matrix where $\psi\in Aut(As^4_4)$. Suppose $\{e^{\prime}_1, e^{\prime}_2,e^{\prime}_3,e^{\prime}_4\}$ be a new basis that obtain by simply multiplying $\psi$ with the basis i.e.,
\begin{align*}
\left(\begin{array}{c}
e^{\prime}_1\\
e^{\prime}_2\\
e^{\prime}_3\\
e^{\prime}_4\\
\end{array}
\right)&=\left(\begin{array}{cccc}
a_{11}&a_{12}&a_{13}&a_{14}\\
a_{21}&a_{22}&a_{23}&a_{24}\\
a_{31}&a_{32}&a_{33}&a_{34}\\
a_{41}&a_{42}&a_{43}&a_{44}\\
\end{array}
\right)^{T}
\left(\begin{array}{c}
e_{1}\\
e_{2}\\
e_{3}\\
e_{4}\\
\end{array}
\right)
\end{align*}
Thus, $\{e^{\prime}_1,e^{\prime}_2,e^{\prime}_3,e^{\prime}_4\}$ can be written as follows

\begin{eqnarray}\label{eq4}
e^{\prime}_1=a_{11}e_1+a_{21}e_2+a_{31}e_3+a_{41}e_4,\\ \nonumber
e^{\prime}_2=a_{12}e_1+a_{22}e_2+a_{32}e_3+a_{42}e_4,\\ \nonumber
e^{\prime}_3=a_{13}e_1+a_{23}e_2+a_{33}e_3+a_{43}e_4,\\ \nonumber
e^{\prime}_4=a_{14}e_1+a_{24}e_2+a_{34}e_3+a_{44}e_4.\nonumber
\end{eqnarray}
Now consider the algebra $As^4_4: e_1e_2 = e_3, e_2e_2 = e_4, e_2e_1 = -e_3$. By applying the new basis (\ref{eq4}) to the algebra we get the following table of multiplications:
\begin{align*}
e^{\prime}_1e^{\prime}_1&=a^2_{21}e_4=0,\\
e^{\prime}_1e^{\prime}_2&=(a_{11}a_{22}-a_{12}a_{21})e_3+a_{21}a_{22}e_4=a_{13}e_1+a_{23}e_2+a_{33}e_3+a_{43}e_4,\\
e^{\prime}_1e^{\prime}_3&=(a_{11}a_{23}-a_{13}a_{21})e_3+a_{21}a_{23}e_4=0,\\
e^{\prime}_1e^{\prime}_4&=(a_{11}a_{24}-a_{14}a_{21})e_3+a_{21}a_{24}e_4=0,\\
e^{\prime}_2e^{\prime}_1&=(a^2_{21}-a_{11}a_{22})e_3+a_{21}a_{22}e_4=-(a_{13}e_1+a_{23}e_2+a_{33}e_3+a_{43}e_4), \\
e^{\prime}_2e^{\prime}_2&=a^2_{22}e_4=a_{14}e_1+a_{24}e_2+a_{34}e_3+a_{44}e_4,\\
e^{\prime}_2e^{\prime}_3&=(a_{21}a_{23}-a_{13}a_{22})e_3+a_{22}a_{23}e_4=0,\\
e^{\prime}_2e^{\prime}_4&=(a_{21}a_{24}-a_{14}a_{22})e_3+a_{22}a_{24}e_4=0,\\
e^{\prime}_3e^{\prime}_1&=(a_{13}a_{21}-a_{11}a_{23})e_3+a_{21}a_{23}e_4=0,\\
e^{\prime}_3e^{\prime}_2&=(a_{13}a_{22}-a_{12}a_{23})e_3+a_{22}a_{23}e_4=0,\\
e^{\prime}_3e^{\prime}_3&=a^2_{23}e_4=0,\\
e^{\prime}_3e^{\prime}_4&=(a_{13}a_{24}-a_{14}a_{23})e_3+a_{23}a_{24}e_4=0,\\
e^{\prime}_4e^{\prime}_1&=(a_{14}a_{21}-a_{11}a_{24})e_3+a_{21}a_{24}e_4=0,\\
e^{\prime}_4e^{\prime}_2&=(a_{14}a_{22}-a_{12}a_{24})e_3+a_{22}a_{24}e_4=0,\\
e^{\prime}_4e^{\prime}_3&=(a_{14}a_{23}-a_{13}a_{24})e_3+a_{23}a_{24}e_4=0,\\
e^{\prime}_4e^{\prime}_4&=a^2_{24}e_4=0.
\end{align*}

Then we have the system $a_{13}=a_{14}= a_{21}=a_{23}=a_{24}=a_{43}=0, \;a_{22}^2=a_{44},a_{33}=a_{11}a_{12}$ and $a_{11},a_{12},a_{31},a_{32},a_{41},a_{42}$ are any. By solving
the system, we obtain the group of automorphism for $As^4_4$ as follows
\begin{align*}
 Aut(As^4_4)&= \left(\begin{array}{cccc}
a_{11}&a_{12}&0&0\\
0&a_{22}&0&0\\
a_{31}&a_{32}&a_{11}a_{22}&0\\
a_{41}&a_{42}&0&a_{22}^2
\end{array}
\right)
,  where \quad a_{11}\neq 0\; or\;a_{22}\neq 0.
\end{align*}
Since $\{a_{11}, a_{31},a_{41}\}$ is a basis of $Aut(As^4_4)$, therefore $dim(Aut(As^4_4)) = 7$.
By applying the similar method for other algebras as shown
above, we get all
the list of automorphism groups as in Theorem \ref{thm6}.
\end{proof}

\begin{cor}\,
\begin{itemize}
	\item The dimensions of the automorphism groups of $2$-dimensional associative algebras range between $1$ and $2$.
		\item The dimensions of the automorphism groups of $3$-dimensional associative algebras range between $2$ and $4$.
	\item The dimensions of the automorphism groups of $4$-dimensional associative algebras range between $1$ and $12$.
	\end{itemize}
\end{cor}

\section*{Conclusion}
In this work, we examined the derivations and automorphism groups of certain finite-dimensional associative algebras over $\mathbb{C}$. The dimension of the space of derivations ranges from 0 to 2 for two-dimensional, 2 to 4 for three-dimensional, and 0 to 12 for four-dimensional complex associative algebras. Similarly, the dimension of the automorphism groups ranges from 1 to 2 for two-dimensional, 2 to 4 for three-dimensional, and 1 to 12 for four-dimensional algebras. These dimensions serve as important invariants in the geometric classification of algebras and have wide applications in various fields.

\section*{Acknowledgement}
The authors thank the anonymous referees for their valuable suggestions and comments.

\section*{ Conflicts of Interest}
The authors declare no conflicts of interest.\\
\cite{a,b,c,d,e,f,g,h,i,j,k,l,m,n,o,p,q,r,s,t,u,v,w,x,y}

\end{document}